\documentclass[preprint]{elsarticle}
\usepackage{cmap} % for russuan cope paste
           \usepackage[T2A]{fontenc}   % Or any other encoding.
           \usepackage[utf8]{inputenc}
           \usepackage[english]{babel}
\usepackage{amsmath}
\usepackage{amsfonts}
\usepackage{amssymb}
\usepackage{amsthm}
\usepackage[pdftex,unicode]{hyperref}
\usepackage{indentfirst} % включить отступ у первого абзаца
\usepackage{amscd}
\usepackage[all,cmtip]{xy}
\usepackage{mathtools}
\def\nfs/{NFS}
\def\cdp/{CDP}
\def\cdpz/{CDP${}_0$}

\def\bbset#1:#2\eeset{\{#1\,:\,#2\}}

\def\bbsett#1:#2\eesett{\{#1\,:\,\text{#2}\}}

\def\ibbset#1:#2\ieeset{(#1)_{#2}}

\newcommand\restrA[2]{{% we make the whole thing an ordinary symbol
  \left.\kern-\nulldelimiterspace % automatically resize the bar with \right
  #1 % the function
  \vphantom{\big|} % pretend it's a little taller at normal size
  \right|_{#2} % this is the delimiter
  }}

\def\pwr#1_#2{#1^{[#2]}}

\def\dddm#1(#2){N_{#1}(#2)}
\def\dddb#1(#2){B_{#1}(#2)}

\def\et(#1){ (#1)}

\def\bitem#1,#2.{ $#2\nrightarrow #1$:\ }
\newtheorem{proposition}{Proposition}
\newtheorem{theorem}{Theorem}
\newtheorem{lemma}{Lemma}
\newtheorem*{lemma*}{Lemma}
\newtheorem{cor}{Corollary}

\theoremstyle{definition}
\newtheorem{definition}{Definition}

\def\oo#1/{$O_{#1}$}

\def\gd/{$G_\delta$}

\def\kwd{}
\def\reftext{}

\begin{document}
\begin{frontmatter}

\title{Extension of mappings from the product of pseudocompact spaces}
\author{Evgenii Reznichenko}

\ead{erezn@inbox.ru}
%\orcid{0000-0002-3384-5806}}
\address{Department of General Topology and Geometry, Mechanics and Mathematics Faculty, M.~V.~Lomonosov Moscow State University, Leninskie Gory 1, Moscow, 199991 Russia}

\begin{abstract}
Let $X$ and $Y$ be pseudocompact spaces and let a function
$\Phi: X\times Y\to\mathbb R$ be separately continuous. The following
conditions are equivalent: (1) there is a dense $G_{\delta}$ subset of
$D\subset Y$ such that $\Phi$ is continuous at every point of
$X\times D$ (Namioka property); (2) $\Phi$ is quasicontinuous; (3)
$\Phi$ extends to a separately continuous function on
$\operatorname{\beta}X\times\operatorname{\beta}Y$. This theorem makes
it possible to combine studies of the Namioka property and generalizations
of the Eberlein-Grothendieck theorem on the precompactness of subsets of
function spaces. We also obtain a characterization of separately continuous
functions on the product of several pseudocompact spaces extending to separately
continuous functions on products of Stone--\v Cech extensions of spaces.
These results are used to study groups and Mal'tsev spaces with separately
continuous operations.
\end{abstract}

\end{frontmatter}

\section{Introduction}
\label{sec1}

Let $X$ and $Y$ be topological spaces. A function
$\Phi: X\times Y\to\mathbb R$ is separately continuous if the functions
$\Phi(\cdot,y): X\to\mathbb R$ and
$\Phi(x,\cdot): Y\to\mathbb R$ are continuous for $x\in X$ and
$y\in Y$. The consideration of separate continuity vis-\`a-vis joint continuity
goes back, at least, to Bare 1899 \cite{baire1899}, whose work is the prototype
of all the subsequent investigations on this subject by many mathematicians.

A function $f: Z\to\mathbb R$ is called \textit{quasi-continuous} if for
every point $z\in Z$, neighborhood $O$ of $f(z)$, neighborhood $W$ of
$z$ there exists a non-empty open $U\subset W$ such that $f(U)\subset O$.

The following continuity conditions for the function $\Phi$ are considered.
\begin{itemize}[($C_{3}$)]
\item[{\textup{($C_{1}$)}}] There is a dense type $G_{\delta}$ subset
$D\subset Y=\overline{D}$ such that $\Phi$ is (jointly) continuous at every
point $(x,y)\in X\times D$ \cite{namioka1974}.
\item[{\textup{($C_{2}$)}}] The function $\Phi$ is quasi-continuous.
\item[{\textup{($C_{3}$)}}] The function $\Phi$ extends to a separately
continuous function
$\widehat{\Phi}: \operatorname{\beta}X\times Y\to\mathbb R$, where
$\operatorname{\beta}X$ is the Stone--\v Cech extension of $X$
\cite{Reznichenko1994}.
\end{itemize}
Clearly, \textup{($C_{1}$)} implies \textup{($C_{2}$)}. Spaces
$X$ and
$Y$ satisfy the \textit{Namioka property} ${\mathcal{N}}( X , Y )$ if for
every separately continuous map $\Phi$ the condition \textup{($C_{1}$)}
 is satisfied \cite{namioka1974}. If the condition \textup{($C_{1}$)}
is satisfied,
then the function $\Phi$ is also said to satisfy the Namioka
property. We say that $(X,Y)$ is a \textit{Grothendieck pair} if for every
separately continuous map $\Phi$ the condition \textup{($C_{3}$)} is satisfied
\cite{Reznichenko1994}.

Note that \cite{Reznichenko1994} gave a different definition:
$(X,Y)$ is a \textit{Grothendieck pair} if for every continuous map
$\varphi: X\to C_{p}(Y)$ the closure of $\varphi( X)$ in $C_{p}(Y)$ is compact,
where $C_{p}(Y)$ is the space of continuous functions on $Y$ in the topology
of pointwise convergence \cite[Definition 1.7]{Reznichenko1994}.
Assertion 1.2 of \cite{Reznichenko1994} implies that these two definitions
are equivalent.

An important special case of the general situation is when the spaces
$X$ and $Y$ are pseudocompact. This case was mainly considered in
\cite{Reznichenko1994,ReznichenkoUspenskij1998}. The main result of this
paper is that in this case, if the function $\Phi$ is separately continuous,
then the conditions \textup{($C_{1}$)}, \textup{($C_{2}$)} and
\textup{($C_{3}$)}
are equivalent (\reftext{Theorem~\ref{t:ef:1}}). This implies that if the spaces
$X$ and $Y$ are pseudocompact, then $(X,Y)$ is a Grothendieck pair if and
only if $X$ and $Y$ satisfy the Namioka property (\reftext{Theorem~\ref{t:fsv:1}}). This theorem allows using the Namioka property theorems
to find Grothendieck pairs and vice versa, the Grothendieck pair theorems
to find pairs of spaces with the Namioka property.

A space $Y$ is called a weakly $pc$-Grothendieck space if any pseudocompact
subspace of $C_{p}(Y)$ has a compact closure in $C_{p}(Y)$
\cite{Arhangelskii1997}. In other words, $Y$ is a weakly $pc$-Grothendieck
space if and only if $(X,Y)$ is a Grothendieck pair for any pseudocompact
space $X$. \reftext{Theorem~\ref{t:fsv:1}} allows one to find new classes of
pseudocompact
weakly $pc$-Grothendieck spaces.

Using \reftext{Theorem~\ref{t:ef:1}} and the results of \cite{Reznichenko2022} in
Section~\ref{sec-fsv}, we obtain a criterion for a function of several
variables on a product of pseudocompact spaces  to extend to a product
of Stone--\v Cech extensions (\reftext{Theorem~\ref{t:fsv:1}} and \reftext{Theorem~\ref{t:fsv:2}}). Using the results of Section~\ref{sec-ef} in 
Section~\ref{sec-maltsev}, we obtain theorems on the continuity of operations
in Mal'tsev groups and spaces.

The terminology follows the books
\cite{Arkhangelskii1992cpbook,EngelkingBookGT}. By spaces we mean Tikhonoff
spaces.

%s2 #&#
\section{Extension of functions from a product of spaces}
%%LEAP%%%\label{sec2}
\label{sec-ef}

%p1 #&#
%
\begin{proposition}%
%%LEAP%%%\label{prop1}
\label{p:ef:1}
Let $X$ and $Y$ be pseudocompact spaces and
$\Phi: X\times Y \to\mathbb R$ be a separately continuous quasi-continuous
function. Then $f$ extends to a separately continuous function
$\widehat{\Phi}:\beta X\times Y \to\mathbb R$.
\end{proposition}
\begin{proof}
Denote by $C$ the set of points in $X\times Y$ at which the function
$\Phi$ is continuous. Denote $C_{y}=\{x\in X\,:\, (x,y)\in C\}$ for
$y\in Y$.
%
%l1 #&#
%
\begin{lemma}%
%%LEAP%%%\label{lem1}
\label{l:ef:1}
The set $C_{y}$ is dense in $X$ for all $y\in Y$.
\end{lemma}
\begin{proof}
Assume the opposite, i.e.
$U'=X\setminus\overline{C_{y'}}\neq\varnothing$ for some
$y'\in Y$. Let us put
\begin{equation*}
\Psi(x,y)=|\Phi(x,y)-\Phi(x,y')|
\end{equation*}
for $(x,y)\in X\times Y$. The function $\Psi$ is non-negative, separately
continuous, quasi-continuous,  and discontinuous 
at points of the set $U'\times\{y'\}$, and $\Psi(x,y')=0$ for $x\in X$. For $O\subset X\times Y$ and
$(x,y)\in X\times Y$ we set
\begin{align*}
\omega_{\Psi}(O) &= \sup\{|\Phi(x_{1},y_{1})-\Phi(x_{2},y_{2})|
\,:\, (x_{1},y_{1}),(x_{2},y_{2})\in O\},
\\
\omega_{\Psi}(x,y) &= \inf\{\omega_{\Psi}(O)\,:\, O
\text{ is a neighborhood of the point }(x,y)\}.
\end{align*}
Let us put
\begin{align*}
F_{n}&=\{(x,y)\in X\times Y \,:\, \omega_{\Psi}(x,y)\geq
\frac{1}{2^{n}}\}, & F_{n}'&=\{ x\in X\,:\, (x,y')\in F_{n}\}
\end{align*}
for $n\in\omega$. The set $F_{n}$ is closed in $X\times Y$ and
$F_{n}'$ is closed in $X$. The set $\bigcup_{n\in\omega} F_{n}$ is the
set of discontinuity points of the function $\Psi$, so
$U'\subset\bigcup_{n\in\omega} F_{n}'$. Since $X$ is a Baire space,
there exists a non-empty open $U\subset U'\cap F_{n}'$ for some
$n\in\omega$. We set $\varepsilon=\frac{1}{3\cdot2^{n}}$ and
\begin{equation*}
M=\{(x,y)\in X\times Y\,:\, \Psi(x,y)> 2\varepsilon\}.
\end{equation*}
Then $U\times\{y'\}\subset\overline{M}$. Let $U_{-1}=U$ and
$V_{-1}=Y$. By induction on $n$ we construct a sequence
\begin{equation*}
(x_{n}, V_{n}, U_{n}, W_{n})_{n\in\omega},
\end{equation*}
where $x_{n}\in U$, $V_{n}\subset Y$ is an open neighborhood of $y'$,
$U_{n}\subset U$ is an open non-empty set, $W_{n}\subset Y$ is an open
non-empty set such that for every $n\in\omega$ the following conditions
are met:
\begin{itemize}
\item[\textup{(1)}] $x_{n}\in U_{n}$ and
$\overline{U_{n}}\subset U_{n-1}$;
\item[\textup{(2)}] $y'\in V_{n}$, $\overline{V_{n}}\subset V_{n-1}$ and
$W_{n}\subset V_{n}$;
\item[\textup{(3)}]
$\Psi(\{x_{n}\}\times V_{n})\subset[0,\varepsilon)$;
\item[\textup{(4)}]
$\Psi(U_{n}\times W_{n})\subset(2\varepsilon,+\infty)$.
\end{itemize}
Let us carry out the construction at the $n$th step. Since
$U\times\{y'\}\subset\overline{M}$, $y'\in V_{n-1}$ and
$U_{n-1}\subset U$, there exists
$(x'',y'')\in M \cap(U_{n-1}\times V_{n-1})$. Then
$\Psi(x'',y'')> 2\varepsilon$. Since the function is quasi-continuous,
$\Psi(U_{n}\times W_{n})\subset(2\varepsilon,+\infty)$ for some non-empty
open $U_{n}\subset\overline{U_{n}}\subset U_{n-1}$ and
$W_{n}\subset\overline{W_{n}}\subset V_{n-1}$. Take
$x_{n}\in U_{n}$. We choose a neighborhood $V_{n}$ of the point $y'$ in
such a way that $\overline{V_{n}}\subset V_{n-1}$ and
$\Psi(\{x_{n}\}\times V_{n})\subset[0,\varepsilon) $.

Let $G=\bigcap_{n\in\omega} U_{n}$. Since $X$ is pseudocompact, then
$G$ is a non-empty closed subset of $X$. Since $Y$ is pseudocompact, the
sequence $(W_{n})_{n\in\omega}$ accumulates to some point
$y_{*}\in Y$. We put $f(x)=\Psi(x,y_{*})$ for $x\in X$. The function
$f: X\to\mathbb R$ is continuous. It follows from (2) that
$y_{*}\in Q=\bigcap_{n\in\omega} V_{n}$. It follows from (4) that
$f(G)\subset[2\varepsilon,+\infty)$. Since $y_{*}\in V_{n}$, it follows
from (3) that $f(x_{n})<\varepsilon$ for $n\in\omega$. Take a neighborhood
$O_{n}$ of the point $x_{n}$ such that $O_{n}\subset U_{n}$ and
$f(O_{n})\subset[0,\varepsilon)$. Since $X$ is pseudocompact, the sequence
$(O_{n})_{n\in\omega}$ accumulates to some point $x_{*}\in G$. Since
$f(O_{n})\subset[0,\varepsilon)$ for $n\in\omega$ we have
$f(x_{*})\leq\varepsilon$. This contradicts the fact that
$x_{*}\in G$ and $f(G)\subset[2\varepsilon,+\infty)$.
\end{proof}
For $x\in X$ and $y\in Y$, denote $\Phi_{y}(x)=\Phi(x,y)$. The function
$\Phi_{y}: X\to\mathbb R$ is continuous and bounded. Let
$\widehat{\Phi}_{y}: \beta X\to\mathbb R$ be a continuous extension of
$\Phi_{y}$. We put
$\widehat{\Phi}(x,y)=\widehat{\Phi}^{x}(y)=\widehat{\Phi}_{y}(x)$ for
$x\in\beta X$ and $y\in Y$. Let us check that the function
$\widehat{\Phi}$ is separately continuous. Let us assume the opposite.
Then $f=\widehat{\Phi}^{\bar x}$ is discontinuous for some
$\bar x\in\beta X$. Let $\bar y\in Y$ be a discontinuity point of
$f$. Without loss of generality, we can assume that $f(\bar y)=0$ and
$\bar y\in\overline{M}$, where $M=f^{-1}([1,+\infty))$. We set
$W=\operatorname{Int}\Phi^{-1}((-\infty,\frac{1}{3}))$ and
$U=\{x\in X\,:\, (x,\bar y)\in W\}$. \reftext{Lemma~\ref{l:ef:1}} implies that
$C_{\bar y}$ is dense in $X$. Hence
$\bar x\in\overline{U}^{\beta X}$. For $y\in M$, let
$U_{y}\subset\beta X$ be an open neighborhood of $\bar x$ such that
$\widehat{\Phi}_{y}(U_{y})\subset(\frac{2}{3},+\infty)$.

Let $V_{-1}=Y$. By induction on ${n\in\omega}$ we construct
$y_{n}\in Y$, $U_{n}$, $V_{n}\ni\bar y$, where $U_{n}$ is an open non-empty
subset of $X$ and $V_{n}$ is open in $Y$. In this case, the following conditions
are met:
\begin{itemize}
\item[\textup{(1)}] $y_{n}\in V_{n-1}\cap M$;
\item[\textup{(2)}] $U_{n}\times V_{n}\subset W$;
\item[\textup{(3)}] $U_{n}\subset\bigcap_{i=0}^{n} U_{y_{i}}$;
\item[\textup{(4)}] $\overline{V_{n}}\subset V_{n-1}$.
\end{itemize}
On the $n$th move we choose $y_{n}\in V_{n-1}\cap M$. Let
$U'=\bigcap_{i=0}^{n} U_{y_{i}}$ and
$(x',\bar y)\in W\cap(U'\times V_{n-1})$. Take open
$U_{n}\subset X$ and $V_{n}\subset Y$, so that
\begin{equation*}
(x',\bar y) \in U_{n}\times V_{n} \subset
\overline{U_{n}\times V_{n}} \subset W \cap(U' \times V_{n-1}).
\end{equation*}
Since the space $X$ is pseudocompact, the sequence $(U_{n})_{n}$ accumulates
to some point $x_{*}\in X$. We set $g=\widehat{\Phi}^{x_{*}}$. Since (3),
we have $g(y_{n})\geq\frac{2}{3}$. Since (1) and the function $g$ is continuous,
there exists a neighborhood $S_{n}$ of the point $y_{n}$ such that
$g(S_{n})\subset(\frac{1}{2},+\infty)$ and
$S_{n}\subset V_{ n-1}$. Since (4) and the space $Y$ is pseudocompact,
 $(S_{n})_{n}$ accumulates to some point
$y_{*}\in G=\bigcap_{n} V_{n}$. The continuity of $g$ implies that
$g(y_{*})\geq\frac{1}{2}$. Condition (2)  implies
$g(y_{*})\leq\frac{1}{3}$. This is a contradiction.
\end{proof}

%t1 #&#
%
\begin{theorem}%
%%LEAP%%%\label{thm1}
\label{t:ef:1}
Let $X$ and $Y$ be pseudocompact spaces and
$\Phi: X\times Y \to\mathbb R$ be a separately continuous function. Denote
\begin{align*}
\varphi_{X}&: X\to C_{p}(Y),\ \varphi_{X}(x)(y) = \Phi(x,y),
\\
\varphi_{Y}&: Y\to C_{p}(X),\ \varphi_{Y}(y)(x) = \Phi(x,y).
\end{align*}
The following conditions are equivalent:
\begin{itemize}[(10)]
\item[\textup{(1)}] there is a dense type $G_{\delta}$ subset
$D\subset Y=\overline{D}$ such that $\Phi$ is continuous at every point
$(x,y)\in X\times D$;
\item[\textup{(2)}] the function $\Phi$ is quasicontinuous;
\item[\textup{(3)}] the closure of $\varphi_{X}(X)$ in $C_{p}(Y)$ is compact;
\item[\textup{(4)}] $\varphi_{X}(X)$ is an Eberlein compactum;
\item[\textup{(5)}] $\Phi$ extends to a separately continuous
function on
$\operatorname{\beta}X\times Y$;
\item[\textup{(6)}] $\Phi$ extends to a separately continuous
function on
$\operatorname{\beta}X\times\operatorname{\beta}Y$;
\item[\textup{(7)}] $\Phi$ extends to a separately continuous
function on
$X\times\operatorname{\beta}Y$;
\item[\textup{(8)}] $\varphi_{Y}(Y)$ is an Eberlein compactum;
\item[\textup{(9)}] the closure of $\varphi_{Y}(Y)$ in $C_{p}(X)$ is compact;
\item[\textup{(10)}] there exists a dense type $G_{\delta}$ subset
$E\subset X=\overline{E}$ such that $\Phi$ is continuous at every point
$(x,y)\in E\times Y$.
\end{itemize}
\end{theorem}
\begin{proof}
The equivalence of conditions from (3) to (9) follows from
\cite[Assertion 1.4]{Reznichenko1994} and
\cite[Proposition
3.1]{ReznichenkoUspenskij1998}. The implications 
$(1)\Rightarrow(3) \Leftarrow(10)$ are obvious. The implication
$(2)\Rightarrow(5)$ is \reftext{Proposition~\ref{p:ef:1}}.

Let us prove $(6)\Rightarrow(1)$. Let
$\widehat{\Phi}: \operatorname{\beta}X\times\operatorname{\beta}Y
\to\mathbb R$ be a separately continuous extension of the function
$\Phi$. The pair of compact spaces $\operatorname{\beta}X$ and
$\operatorname{\beta}Y$ satisfy the Namioka property
${\mathcal{N}}( \operatorname{\beta}X , \operatorname{\beta}Y )$
\cite{namioka1974}. Hence there is a dense type $G_{\delta}$ subset
$D'\subset\beta Y=\overline{D'}$ such that $\widehat{\Phi}$ is continuous
at every point $(x,y)\in\beta X\times D$. Since $Y$ is pseudocompact,
$D=Y\cap D'$ is dense in $Y$ and of type $G_{\delta}$ in $Y$. Then
$\Phi$ is continuous at every point $(x,y)\in X\times D$.

The implication $(6)\Rightarrow(10)$ follows from the implication
$(6)\Rightarrow(1)$.
\end{proof}

A space $X$ is called \textit{pc-Grothendieck} (\textit
{pe-Grothendieck}) if
any pseudocompact subspace of $C_{p}(X)$ is an (Eberlein) compact set. A
space $X$ is called \textit{weakly pc-Grothendieck} if any
pseudocompact subspace
of $C_{p}(X)$ has a compact closure in $C_{p}(X)$)
\cite{Arhangelskii1997}.

%t2 #&#
%
\begin{theorem}%
%%LEAP%%%\label{thm2}
\label{t:ef:2}
Let $X$ be a pseudocompact space and let $Y\subset C_{p}(X)$ be pseudocompact.
The following conditions are equivalent:
\begin{itemize}
\item[\textup{(1)}] $\overline{Y}$ is compact;
\item[\textup{(2)}] $Y$ is compact;
\item[\textup{(3)}] $Y$ is an Eberlein compactum;
\item[\textup{(4)}] $Y$ is weakly pc-Grothendieck;
\item[\textup{(5)}] $\{f\in Y\,:\,$ the restrictions to $Y$ of the
topologies of
pointwise and uniform convergence coincide at $f\}$ dense in $Y$;
\item[\textup{(6)}] $\{f\in Y\,:\, \chi(f,Y)\leq\omega\}$ is dense in $Y$;
\item[\textup{(7)}] $\{f\in Y\,:\, \pi\chi(f,Y)\leq\omega\}$
is dense in $Y$.
\end{itemize}
\end{theorem}
\begin{proof}
Let $\Phi: X\times T\to\mathbb R,\ (x,f)\mapsto f(x)$ and let
$\varphi_{Y}: Y\to C_{p}(X),\ \varphi_{Y}(y)(x) = \Phi(x,y)$. Then
$\varphi_{Y}$ is the identity mapping of $Y$ onto $Y$.

$(1)\Leftrightarrow(2)\Leftrightarrow(3)$. These implications follow
from \reftext{Theorem~\ref{t:ef:1}} (8) and (9).

$(2)\Rightarrow(4)$. This implication follows from the Asanov--Velichko
theorem \cite{AsanovVelichko1981} (see also
\cite[III.4.1. Theorem]{Arkhangelskii1992cpbook}).

$(4)\Rightarrow(1)$. Since $Y$ is pc-Grothendieck,
$\varphi_{X}(X)$ is compact. \reftext{Theorem~\ref{t:ef:1}} (9) implies that
$\overline{Y}$ is compact.

$(2)\Rightarrow(5)$. We identify $C(X)$ and
$C(\operatorname{\beta}X)$ in a natural way. Then by the Haydon theorem
\cite{Haydon1972} the restrictions to $Y$ of the topologies of
$C_{p}(X)$ and $C_{p}(\operatorname{\beta}X)$ coincide. Hence it suffices
to prove the implication for compact $X$, and for compact $X$ this implication
is exactly the same as the Namioka theorem
\cite[Theorem 2.31]{namioka1974}.

Obviously $(5)\Rightarrow(6)\Rightarrow(7)$.

$(7)\Rightarrow(2)$. It follows from \reftext{Corollary~\ref{c:grot-q:1}} that
the function $\Phi$ is quasi-continuous. \reftext{Theorem~\ref{t:ef:1}} (8) implies
that $Y$ is compact.
\end{proof}

%s3 #&#
\section{Quasi-continuous functions}
%%LEAP%%%\label{sec3}
\label{sec-grot-q}

Let us define topological games $G_{g}(y_{*},Y)$ and
$G_{\tilde g}(y_{*},Y)$ for the space $Y$ and $y\in Y$
\cite{Gruenhage1976,DolezalMoors2017}. Players $\alpha$ and
$\beta$ are playing. On the $n$th move, player $\alpha$ chooses
\begin{itemize}
\item an open neighborhood $W_{n}\subset Y$ of point $y_{*}$ in the game
$G_{g}(y_{*},Y)$;
\item an open non-empty set $W_{n}\subset Y$ in the game
$G_{\tilde g}(y_{*},Y)$.
\end{itemize}
Player $\beta$ chooses $y_{n}\in W_{n}$. Player $\alpha$ wins if
$y_{*}\in\overline{\{y_{n}\,:\,{n\in\omega}\}}$.

A point $x\in X$ is called a \textit{$W$-point} (\textit{$\widetilde
{W}$-point})
if 
%AQF
the player $\alpha$ has a winning strategy in the game
$G_{g}( y_{*},Y )$ ($G_{\tilde g}(y_{*},Y)$). A space $X$ is called a
$W$-space ($\widetilde{W}$-space) if every point in $Y$ is a $W$-point
($\widetilde{W}$-point) \cite{Gruenhage1976,DolezalMoors2017}.

%p2 #&#
%
\begin{proposition}[{\cite[Theorem 11]{DolezalMoors2017}}]%
%%LEAP%%%\label{prop2}
\label{p:grot-q:7}
Suppose that $X$ and $Y$ are topological spaces and
$\Phi: X \times Y \to\mathbb R$ is a separately continuous function.
If $X$ is a Baire space and $Y$ is a $\widetilde{W}$-space, then
$\Phi$ is quasi-continuous.
\end{proposition}

%p3 #&#
%
\begin{proposition}%
%%LEAP%%%\label{prop3}
\label{p:grot-q:8}
Suppose that $X$ and $Y$ are topological spaces,
$\Phi: X \times Y \to\mathbb R$ is a separately continuous function,
$M\subset Y\subset\overline{M}$, and a function
\begin{equation*}
\left.\Phi\right|_{X\times M}: X\times M \to\mathbb R%
\end{equation*}
is quasi-continuous. Then $\Phi$ is quasi-continuous.
\end{proposition}
\begin{proof}
Let $(x',y')\in X \times Y$, and let $W=U\times V$ be a neighborhood of
$(x',y')$ and $O\subset\mathbb R$ be a neighborhood of
$\Phi(x',y')$. Let $S$ be a neighborhood of the point
$\Phi(x',y')$ such that $\overline{S}\subset O$. Since $\Phi$ is separately
continuous, we have $\Phi(x',y'')\in S$ for some $y''\in M\cap V$. Since
$\Psi=\left.\Phi\right|_{X\times M}$ is quasi-continuous, we have
\begin{equation*}
\Psi(U'\times(M\cap V'))\subset S
\end{equation*}
for some non-empty open $U'\times V'\subset U\times V$. Then
$\Phi(U'\times V')\subset\overline{S}\subset O$.
\end{proof}

From  \reftext{Propositions~\ref{p:grot-q:7} and \ref{p:grot-q:8}} the following
statement follows.

%p4 #&#
%
\begin{proposition}%
%%LEAP%%%\label{prop4}
\label{p:grot-q:9}
Suppose that $X$ and $Y$ are topological spaces,
$Z\subset Y=\overline{Z}$ and $\Phi: X \times Y \to\mathbb R$ is a separately
continuous function. If $X$ is a Baire space and $Z$ is a
$\widetilde{W}$-space, then $\Phi$ is quasi-continuous.
\end{proposition}

A space with  countable character is a $W$-space
\cite{Gruenhage1976} and a space with  countable $\pi$-character is a
$\widetilde{W}$-space \cite[Proposition 33]{rezn2022-1}.

%c1 #&#
%
\begin{cor}%
%%LEAP%%%\label{cor1}
\label{c:grot-q:1}
Suppose that $X$ and $Y$ are topological spaces and
$\Phi: X \times Y \to\mathbb R$ is a separately continuous function.
If $X$ is a Baire space and
$\{y\in Y\,:\, \pi\chi(y,Y)\leq\omega\}$ is dense in $Y$, then
$\Phi$ is quasi-continuous.
\end{cor}

%s4 #&#
\section{Pseudocompact pc-Grothendieck spaces}
%%LEAP%%%\label{sec4}
\label{sec-grot}

We call a space $X$ a \textit{pf-space} if every pseudocompact
$Y\subset X$ has points with a countable base of neighborhoods. A space
$X$ is called a \textit{pf-Grothendieck} space if $C_{p}(X)$ is a pf-space.

%p5 #&#
%
\begin{proposition}%
%%LEAP%%%\label{prop5}
\label{p:grot:0}
Let $X$ be a pf-space. Then for every pseudocompact $Y\subset X$, the set
$\{y\in Y\,:\, \pi\chi(y,Y)\leq\omega\}$ is dense in $Y$.
\end{proposition}
\begin{proof}
Let $U\subset Y$ be a non-empty set open in $Y$. There is a non-empty set
$V\subset U$ open in $Y$ such that $V\subset S \subset U$, where
$S=\overline{V}\cap Y$. The subspace $S$ is pseudocompact. Let
$(U_{n})_{n}$ be a countable base of some point $s\in S$ in the space
$S$. Let $V_{n}$ be the interior of the set $U_{n}$ in the space $Y$. Then
$(V_{n})_{n}$ is a countable $\pi$-base of the point $s\in Y$ in the space
$Y$.
\end{proof}

%t3 #&#
%
\begin{theorem}%
%%LEAP%%%\label{thm3}
\label{t:grot:1}
Let $X$ be a pseudocompact space. The following conditions are equivalent:
\begin{enumerate}[(6)]
\item[(1)]$X$ is weakly pc-Grothendieck;
\item[(2)]
$X$ is  pc-Grothendieck;
\item[(3)]$X$ is pe-Grothendieck;
\item[(4)]$X$ is pf-Grothendieck;
\item[(5)] for any pseudocompact $Y\subset C_{p}(X)$, one of the equivalent
conditions of \reftext{Theorem~\ref{t:ef:2}} is satisfied:
\begin{enumerate}[(g)]
\item[(a)]$\overline{Y}$ is compact;
\item[(b)]$Y$ is compact;
\item[(c)]$Y$ is an Eberlein compactum;
\item[(d)]$Y$ is weakly pc-Grothendieck;
\item[(e)]$\{f\in Y\,:\,$ the restrictions to $Y$ of the topologies of pointwise
and uniform convergence coincide at $f\}$ is dense in $Y$;
\item[(f)]$\{f\in Y\,:\, \chi(f,Y)\leq\omega\}$ is dense in $Y$;
\item[(g)]$\{f\in Y\,:\, \pi\chi(f,Y)\leq\omega\}$ is dense in $Y$;
\end{enumerate}
\item[(6)] for any pseudocompact space $Y$ one of the following equivalent conditions
is satisfied:
\begin{enumerate}[(d)]
\item[(a)]$X$ and $Y$ form a Grothendieck pair;
\item[(b)]$Y$ and $X$ form a Grothendieck pair;
\item[(c)]$X$ and $Y$ satisfy the Namioka property;
\item[(d)]$Y$ and $X$ satisfy the Namioka property.
\end{enumerate}
\end{enumerate}
\end{theorem}
\begin{proof}
The equivalence of the conditions in (5) follows from \reftext{Theorem~\ref{t:ef:2}}. The equivalence
of the conditions in (6) follows from \reftext{Theorem~\ref{t:ef:1}}. Conditions (1),
(2) and (3) are exactly conditions (a), (b) and (c) in (5). From (f) in condition
(5) follows (4). \reftext{Proposition~\ref{p:grot:0}} and (4) imply (g) in condition
(5). Obviously, condition (a) in (5) is equivalent to (b) in (6).
\end{proof}

We are primarily interested in pseudocompact pc-Grothendieck spaces; this
class was denoted as $\mathcal{L}$ in \cite{Reznichenko1994} and
\cite{Korovin1992} and plays an important role in the study of groups with
topology.

%s5 #&#
\section{Functions of several variables}
%%LEAP%%%\label{sec5}
\label{sec-fsv}

Let $\{X_{\alpha}\,:\, \alpha\in A\}$ be a family of sets, $Y$ a set,
$X=\prod_{\alpha\in A} X_{\alpha}$, $\Phi: X \to Y$ a mapping,
$B\subset A$, and 
$\bar x= (x_{\alpha})_{\alpha\in A\setminus B}\in\prod_{\alpha
\in A\setminus B} X_{\alpha}$. Let us define the mapping
\begin{equation*}
r(\Phi,X,\bar x): \prod_{\alpha\in B} X_{\alpha}\to Y,\ (x_{
\alpha})_{\alpha\in B}\mapsto\Phi(( x_{\alpha})_{\alpha\in A}).
\end{equation*}

%d1 #&#
%
\begin{definition}
%%LEAP%%%\label{defn1}
\label{d:fsv:1}
Let $ A $ be a set, let $ \{X_{\alpha}\,:\, \alpha\in A\} $ be a family
of spaces, $Y$ be a space, and let
$ X = \prod_{\alpha\in A}X_{\alpha}$. Suppose given a map
$ \Phi: X \to Y $ and a positive integer $ n $.
\begin{itemize}
\item(Definition 3.25 of \cite{ReznichenkoUspenskij1998}) The map
$ \Phi$ is \textit{$ n $-separately continuous} iff
$ r (\Phi, X, \bar x)$ is continuous for each $ B \subset A $ with
$ | b | \le n $ and any
$ \bar x\in\prod_ {\alpha\in A \setminus B} X_ {\alpha}$.
\item(Definition 1 of \cite{Reznichenko2022}) The map $ \Phi$ is
\textit{$ n $-$ \beta$-extendable} iff $ g = r (\Phi, X, \bar x)$ extends
to a separately continuous map
$ \hat g: \prod_ {\alpha\in B} \beta X_ {\alpha}\to\beta Y $ for each
$ B \subset A $ with $ | b | \le n $ and any
$ \bar x\in\prod_ {\alpha\in A \setminus B} X_ {\alpha}$.
\item The map $ \Phi$ is \textit{$n$-quasicontinuous} iff
$ r (\Phi, X, \bar x)$ is quasicontinuous for each $ B \subset A $ with
$ | b | \le n $ and any
$ \bar x\in\prod_ {\alpha\in A \setminus B} X_ {\alpha}$.
\end{itemize}
\end{definition}

Separately continuous maps are exactly $ 1 $-separately continuous maps.

A space $X$ is Dieudonn\'e complete if it admits a compatible complete
uniformity. For a space $X$ the Dieudonn\'e completion $\mu X$ can be defined
as the smallest Dieudonn\'e complete subspace of
$\operatorname{\beta}X$ containing $X$. If X is pseudocompact, then
$\beta X$ = $\mu X$, and every continuous map $f: \beta X \to Y$ has an
extension $\hat f: \beta X \to\mu Y$.

Lemma 3.7 of \cite{ReznichenkoUspenskij1998} implies that if $\Phi$ is an
$n$-$\beta$-extendable function (that is, $Y=\mathbb R$), then
$\hat g(X)\subset\mathbb R$, where $g$ is defined in \reftext{Definition~\ref{d:fsv:1}}.

%t4 #&#
%
\begin{theorem}
%%LEAP%%%\label{thm4}
\label{t:fsv:1}
Let $X_{1}$, $X_{2}$, ..., $X_{n}$ be pseudocompact spaces and let
$ \Phi: \prod_{i = 1}^{n} X_{i} \to\mathbb R$ be a separately continuous
function. The following conditions are equivalent.
\begin{itemize}
\item[\textup{(1)}] The function $\Phi$ extends to a separately continuous
function
$\widehat{\Phi}: \prod_ {i = 1} ^ {n} \beta X_{i} \to\mathbb R$.
\item[\textup{(2)}] The function $\Phi$ is $2$-$\beta$-extendable.
\item[\textup{(3)}] The function $\Phi$ is $2$-quasicontinuous.
\end{itemize}
\end{theorem}
\begin{proof}
The equivalence $(1)\Leftrightarrow(2)$ follows from
\cite[Theorem 2]{Reznichenko2022}. The equivalence
$(2)\Leftrightarrow(3)$ follows from \reftext{Theorem~\ref{t:ef:1}}.
\end{proof}

\reftext{Theorem~\ref{t:fsv:1}} and \cite[Lemma 3.7]{ReznichenkoUspenskij1998} imply
the following assertion.

%t5 #&#
%
\begin{theorem}
%%LEAP%%%\label{thm5}
\label{t:fsv:2}
Let $X_{1}$, $X_{2}$, ..., $X_{n}$ be pseudocompact spaces, let $Y$ be a space,
and let $ \Phi: \prod_{i = 1} ^{n} X_{i} \to Y$ be a separately continuous
map. The following conditions are equivalent.
\begin{itemize}
\item[\textup{(1)}] The map $\Phi$ extends to a separately
continuous map
$\widehat{\Phi}: \prod_ {i = 1} ^ {n} \beta X_{i} \to\mu Y$.
\item[\textup{(2)}] The map $\Phi$ is $2$-$\beta$-extendable.
\item[\textup{(3)}] The map $\Phi$ is $2$-quasicontinuous.
\end{itemize}
\end{theorem}

%t6 #&#
%
\begin{theorem}[{\cite[Theorem 3.15.]{ReznichenkoUspenskij1998}}]
%%LEAP%%%\label{thm6}
\label{t:fsv:3}
Let $X_{1}$, $X_{2}$, ..., $X_{n}$ be pseudocompact spaces such that
$(X_{i}, X_{j})$ is a Grothendieck pair for all distinct $i, j$, let $Y$ be
a space, and let $ \Phi: \prod_ {i = 1} ^ {n} X_{i} \to Y$ be a separately
continuous map. Then the map $\Phi$ extends to a separately continuous map
$\widehat{\Phi}: \prod_ {i = 1} ^ {n} \beta X_{i} \to\mu Y$.
\end{theorem}

%c2 #&#
%
\begin{cor}
%%LEAP%%%\label{cor2}
\label{c:fsv:1}
Let $X$ be a pseudocompact pc-Grothendieck space, let $Y$ be a space, and let
$ \Phi: X^{n} \to Y$ be a separately continuous map. Then map
$\Phi$ extends to a separately continuous map
$\widehat{\Phi}: \beta X^{n} \to\mu Y$.
\end{cor}

%s6 #&#
\section{Pseudocompact groups and spaces with a Mal'tsev operation}
%%LEAP%%%\label{sec6}
\label{sec-maltsev}

A group with a topology is called \textit{semitopological} if multiplication
in the group is separately continuous.

%t7 #&#
%
\begin{theorem}
%%LEAP%%%\label{thm7}
\label{t:maltsev:1}
Let $G$ be a pseudocompact semitopological group. The following conditions
are equivalent.
\begin{itemize}
\item[\textup{(1)}] The group $G$ is a topological group.
\item[\textup{(2)}] The multiplication
\begin{equation*}
{\mathfrak{m}}: G\times G\to G,\ (g,h)\mapsto gh
\end{equation*}
in the group $G$ extends to a separately continuous mapping
\begin{equation*}
\hat{{\mathfrak{m}}}: \beta G\times\beta G\to\beta G.
\end{equation*}
\item[\textup{(3)}] The multiplication ${\mathfrak{m}}$ extends to a
separately
continuous mapping $\hat{{\mathfrak{m}}}: (\beta G)^{2}\to\beta G$ and
$(\beta G,\hat{{\mathfrak{m}}})$ is a topological group.
\item[\textup{(4)}] The multiplication ${\mathfrak{m}}$ is quasi-continuous.
\end{itemize}
\end{theorem}
\begin{proof}
$(1)\Rightarrow(3)$ This implication follows from the Comfort--Ross theorem
\cite[Theorem 4.1]{ComfortRoss1966}.

$(3)\Rightarrow(2)$ Obvious.

$(2)\Rightarrow(1)$ This implication follows from Theorem 2.2 and Assertion
2.1 in \cite{Reznichenko1994}.

$(2)\Leftrightarrow(4)$ This implication follows from  \reftext{Theorem~\ref{t:fsv:2}}.
\end{proof}

The implication $(1)\Rightarrow(4)$ can also be proved by using results of
\cite{moors2017} or \cite{rezn2022-1}.

A \textit{Mal'tsev operation} on a set $X$ is a map $M: X^{3}\to X$ satisfying
the identity $M(x,y,y)=M(y,y,x)=x$ for all $x,y\in X$. A space is
\textit{Mal'tsev}
if it admits a continuous Mal'tsev operation.

%t8 #&#
%
\begin{theorem}
%%LEAP%%%\label{thm8}
\label{t:maltsev:2}
Let $X$ be a pseudocompact space with a separately continuous Mal'tsev operation 
$M$. The following conditions are equivalent.
\begin{itemize}
\item[\textup{(1)}] The Mal'tsev operation $M$ extends to a
separately continuous
mapping $\widehat{M}: (\beta X)^{3}\to\beta X$.
\item[\textup{(2)}] The Mal'tsev operation $M$ extends to a
separately continuous
map $\widehat{M}: (\beta X)^{3}\to\beta X$ and $\widehat{M}$ is a Mal'tsev
operation.
\item[\textup{(3)}] The Mal'tsev operation $M$ is $2$-quasicontinuous.
\end{itemize}
If any of the above conditions is satisfied, then $\beta X$ is a Dugundji
compactum.
\end{theorem}
\begin{proof}
$(1)\Rightarrow(2)$ For $x,y,z\in\beta X$ we set
$f_{x}(y,z)=\widehat{M}(x,y,z)$ and $g_{y}(x)= \widehat{M}(x,y,y)$. If
$y\in X$, then $g_{y}(x)=x$ for all $x\in X$. Since the mapping
$g_{y}$ is continuous, then $g_{y}(x)=x$ for all $x\in\beta X$. Hence
$f_{x}(y,y)=x$ for all $x\in\beta X$ and $y\in X$.
Proposition 3.12 of \cite{ReznichenkoUspenskij1998} implies that
$f_{x}(y,y)=x$ for all $y\in\beta X$. We have proved the identity
$\widehat{M}(x,y,y)=x$ for $x,y\in\beta X$. The identity
$\widehat{M}(y,y,x)=x$ is proved similarly.

$(2)\Rightarrow(1)$ Obvious.

$(1)\Leftrightarrow(3)$ This implication follows from \reftext{Theorem~\ref{t:fsv:2}}.

It follows from (2) that $\beta X$ is a compact space  with a separately continuous
Mal'tsev operation $\widehat{M}$. Compact spaces with separately continuous
Mal'tsev operation are Dugundji compact sets
\cite[Theorem 1.8]{ReznichenkoUspenskij1998}.
\end{proof}

\reftext{Corollary~\ref{c:fsv:1}}, \reftext{Theorems~\ref{t:maltsev:1} and \ref{t:maltsev:2}} imply the following assertions.

%c3 #&#
%
\begin{cor}
%%LEAP%%%\label{cor3}
\label{c:maltsev:1}
Let $G$ be a pseudocompact pc-Grothendieck semitopological group. Then
$G$ is a topological group.
\end{cor}

%c4 #&#
%
\begin{cor}
%%LEAP%%%\label{cor4}
\label{c:maltsev:2}
Let $X$ be a pseudocompact pc-Grothendieck space with a separately continuous
Mal'tsev operation $M$. Then the Mal'tsev operation $M$ extends to a separately
continuous Mal'tsev operation
$\widehat{M}: (\beta X)^{3}\to\beta X$ and $\beta X$ is a Dugundji compactum.
\end{cor}

\bibliographystyle{elsarticle-num}
\bibliography{sqcsc}

\end{document}